\documentclass[12pt,reqno]{amsart}

\usepackage{amsthm,amssymb,amstext,amscd,euscript,mathrsfs,amsfonts,amsbsy,amsxtra,latexsym,amsmath}
\usepackage{fullpage}
\usepackage[english]{babel}
\usepackage[latin1]{inputenc}
\usepackage{verbatim}
\usepackage{graphicx} 
\usepackage[all]{xy} 
\usepackage{tikz}
\usepackage{enumitem}
\usepackage{bbm}
\usepackage{marvosym}
\numberwithin{figure}{section} 
\allowdisplaybreaks

\newcommand{\field}[1]{\mathbb{#1}}

\newcommand{\C}{\field{C}}

\numberwithin{equation}{section}
\newtheorem{theorem}{\textbf{Theorem}}
\numberwithin{theorem}{section}
\newtheorem{corollary}[theorem]{\textbf{Corollary}}

\newtheorem{proposition}[theorem]{\textbf{Proposition}}
\theoremstyle{definition}

\newtheorem{remark}{Remark}[section]
\newtheorem{example}[theorem]{Example}

\newcommand{\bea}{\begin{eqnarray}} 
\newcommand{\eea}{\end{eqnarray}} 
\newcommand{\be}{\begin{equation}} 
\newcommand{\ee}{\end{equation}} 
\newcommand{\benn}{\begin{equation*}} 
\newcommand{\eenn}{\end{equation*}}

\title[short title]{On integral structure types}
\author{James Fullwood}
\address{School of Mathematical Sciences\\Shanghai Jiao Tong University\\ 800 Dongchuan Road, Shanghai, China}
\email{fullwood@sjtu.edu.cn}
\begin{document}

\maketitle
\begin{abstract}
We introduce \emph{integral structure types} as a categorical analogue of virtual combinatorial species. Integral structure types then categorify power series with possibly negative coefficients in the same way that combinatorial species categorify power series with non-negative rational coefficients. The notion of an operator on combinatorial species naturally extends to integral structure types, and in light of their `negativity' we define the notion of the \emph{commutator} of two operators on integral structure types. We then extend integral structure types to the setting of stuff types as introduced by Baez and Dolan, and then conclude by using integral structure types to give a combinatorial description for Chern classes of projective hypersurfaces.  
\end{abstract}

\section{Introduction}\label{intro}
The theory of combinatorial species was first introduced by Joyal in 1981 as a sort of categorification of generating functions in enumerative combinatorics, where formal power series corresponding to generating functions are replaced by a category of structures one may put on finite sets \cite{Joyal1}. In particular, a combinatorial species -- which we will refer to as a \emph{structure type} -- is a covariant functor 
\[
F:\mathfrak{Fin}_0\to \mathfrak{Fin},
\]
where $\mathfrak{Fin}_0$ denotes the category of finite sets with only bijections as morphisms, and $\mathfrak{Fin}$ denotes the category of finite sets with arbitrary functions between them. Structure types then form the objects of a category, with morphisms corresponding to natural transformations of functors. The image $F(S)$ of a finite set $S$ under $F$ is thought of as the set of `$F$-structures' one may put on $S$. For example, typical structures one may encounter on finite sets include simple graphs, linear orderings, cycles and partitions. Given a structure type $F$, functoriality implies that the number of $F$-structures one may put on a finite set $S$ only depends on the cardinality of $S$, so that it makes sense to `count' the number of $F$-structures one may put on a set with $n$ elements for every $n\in \mathbb{N}$. We may then associate with $F$ a sequence of non-negative integers $F_n$, where $F_n$ denotes the number of $F$-structures one may put on an $n$-element set.  As such, we denote by $F(z)$ the formal power series
\[
F(z)=\sum_{n=0}^{\infty}\frac{F_n}{n!}z^n,
\]
so that $F(z)$ is the exponential generating series for $F$-structures. It turns out that we can add and multiply structure types in such a way that corresponds to categorical lifts of the usual monoidal operations of addition and multiplication of power series. In particular, if we denote addition and multiplication of structure types by $+$ and $\cdot$ respectively, we have
\[
(F+G)(z)=F(z)+G(z), \quad \text{and} \quad (F\cdot G)(z)=F(z)\cdot G(z),
\]
so that the map which takes a structure type to its associated exponential generating series obeys the same properties of the cardinality function on finite sets with respect to product and coproduct. The category of structure types is then a categorification of formal power series with non-negative rational coefficients, as it replaces equations between power series by isomorphisms in the category of structure types. 

For a typical example of how an isomorphism between structure types can yield an elegant solution to a counting problem, consider the structure of binary rooted trees, which we denote by $B$. Then if we denote the structure of being a 1-element set by $Z$, there exists the following isomorphism of structure types 
\[
B\cong\bold{1}+Z\cdot B^2,
\]
where $\bold{1}$ denotes the identity with respect to product of structure types (which is the structure of being a set with no elements). Taking cardinalities then yields a non-trivial identity in power series, namely
\[
B(z)=1+zB(z)^2,
\] 
which after solving for $B(z)$ and taking a Taylor expansion yields
\[
B(z)=\sum_{n=0}^{\infty}\frac{1}{n+1}\binom{2n}{n}z^n.
\]

To extend the relationship between structure types and power series to include power series with possibly negative coefficients, the theory of `virtual species' was introduced and studied in papers \cite{Joyal2}\cite{Yeh1}, which we refer to as \emph{virtual structure types}. The construction of virtual structure types is analogous to constructing the integers from $\mathbb{N}$, i.e., virtual structure types are equivalence classes of pairs of structure types $(F,G)$, where $(F,G)\sim (F',G')$ if and only if $F+G'\cong F'+G$. The associated power series of a virtual structure type $(F,G)$ is then given by
\[
(F,G)(z)=F(z)-G(z).
\]

While virtual structure types extend the dictionary between power series and combinatorial structures on finite sets, in taking equivalence classes of pairs of structure types we lose the categorical feature of structure types, as virtual structure types form a commutative ring as opposed to a category. The purpose of this note is then to modify the notion of virtual structure type in such a way that we end up with a category as opposed to a ring. What we end up with is an extension of the category of structure types whose objects we refer to as \emph{integral structure types}, which may be viewed as a categorification of power series with possibly negative coefficients.

In what follows, we first review the theory of structure types in \S\ref{ST}. We then introduce a category $\mathfrak{Int}$ in \S\ref{IS} which we refer to as the category of \emph{integral sets}, which is an extension of the category $\mathfrak{Fin}$ to include sets with negative cardinality. In \S\ref{IST} we use the category $\mathfrak{Int}$ to define integral structure types. In particular, we define an integral structure type to be a functor 
\[
\Phi:\mathfrak{Fin}_0\to \mathfrak{Int},
\]
so that integral structure types are constructed by simply enlarging the target category of structure types. The category of structure types then naturally embeds in the category of integral structure types, and the monoidal structures on structure types carry over in a straightforward manner to monoidal structures on integral structure types and their associated power series. In \S\ref{OIST} we define \emph{operators} on integral structure types, which naturally leads to the notion of the \emph{commutator} $[\mathscr{A},\mathscr{B}]$ of two operators on integral structure types $\mathscr{A}$ and $\mathscr{B}$. With a view towards applying the theory of structure types to mathematical structures in quantum mechanics, Baez, Dolan and Morton have generalized structure types to what they refer to as `stuff types' \cite{BaezDolan}\cite{Morton}, and in \S\ref{STF} we show how integral structure types may be generalized to `integral stuff types'. We then conclude in \S\ref{CCC} by using operators on integral structure types to give a combinatorial description for Chern classes of projective hypersurfaces. In particular, the topological Euler characteristic of a hypersurface in complex projective space $\mathbb{P}^n$ is a polynomial in its degree, which we refer to as the $n$th \emph{Euler polynomial}. We then view the $n$th Euler polynomial as the generating series of a certain integral structure type, and then define an operator on integral structure types such that the associated generating series of successive iterations of this operator on the structure type corresponding to the $n$th Euler polynomial yields the total Chern class of projective hypersurfaces.

\section{Structure types}\label{ST}      
We recall that a \emph{structure type} is a covariant functor
\[
F:\mathfrak{Fin}_0\to \mathfrak{Fin},
\]
where $\mathfrak{Fin}_0$ is the category of finite sets but with only bijections as morphisms (so that $\mathfrak{Fin}_0$ is in fact a groupoid), and $\mathfrak{Fin}$ is the category of finite sets with arbitrary maps between them. We denote the associated functor category by $\mathfrak{Fin}[Z]$, and recall that with every structure type $F$ we associate a formal power series given by
\[
F(z)=\sum_{n=0}^{\infty}\frac{F_n}{n!}z^n,
\]
where $F_n$ is the number of $F$-structures one may put on an $n$-element set. We then refer to $F(z)$ as the \emph{generating series} of $F$. Addition and multiplication of structure types may be defined in such a way that is consistent with addition and multiplication of power series. In particular, the \emph{sum} of two structure types $F$ and $G$ is denoted by $F+G$, and is given by
\[
(F+G)(S)=F(S)\sqcup G(S).
\]
The structure type $F+G$ then acts on a finite set $S$ by putting either an $F$-structure or a $G$-structure on $S$, and the identity object with respect to $+$ is the structure type $\bold{0}$ of being a set with infinitely many elements (so that $\bold{0}(S)=\varnothing$ for all finite sets $S$). The \emph{product} of two structure types $F$ and $G$, which we denote by $F\cdot G$, is given by
\[
(F\cdot G)(S)=\{(F(S_1),G(S_2)) \hspace{1mm}| \hspace{1mm} S=S_1\sqcup S_2\},
\]
so that $F\cdot G$ acts on a finite set $S$ by partitioning it into two disjoint subsets, and putting an $F$-structure on what was labeled the first subset and a $G$-structure on what was labeled the second one. The identity object with respect to product of structure types is denoted $\bold{1}$, which is the structure of being a set with no elements, so that 
\[
\bold{1}(S)=\begin{cases} \{\varnothing\} \quad \text{if $S$ is the empty set} \\\hspace{0.19cm} \varnothing \quad \quad \quad \text{otherwise} \\ \end{cases}
\]
It is then straightforward to show
\[
(F+G)(z)=F(z)+G(z), \quad \text{and} \quad (F\cdot G)(z)=F(z)\cdot G(z),
\]
so that the map which takes a structure type to its exponential generating series is a decategorification map for structure types. As such, sums and products of structure types are categorifications of addition and multiplication of formal power series. 

\begin{remark}
There are at least six monoidal operations on structure types which admit useful combinatorial descriptions (see \cite{Yorgey}, Chapter~4), as well as other notions of generating series associated with structure types. For our purposes however, we only need the monoidal operations of sum and product, and from our perspective the notion of exponential generating series seems the most natural for the theory of structure types.  
\end{remark}

\begin{example}\label{e1}
Denote by $Z^k$ the structure type of being a totally ordered $k$-element set. So for example if $S=\{a,b,c\}$, then
\[
Z^3(S)=\{(a,b,c),(a,c,b),(b,a,c),(b,c,a),(c,a,b),(c,b,a)\}.
\]
By definition, $Z^k$ of a set with cardinality not equal to $k$ is the empty set, and if $S$ denotes a set of cardinality $k$ then $\chi(Z^k(S))=k!$,
so
\[
Z^k(z)=\sum_{n=0}^{\infty}\frac{(Z^k)_n}{n!}z^n=\frac{k!}{k!}z^k=z^k.
\]
Now denote by $C^k$ the structure of being a cycle of length $k$, so that if $S$ is a set with $k$-elements then $C^k(S)=Z^k(S)/\mathbb{Z}_k$, where $\mathbb{Z}_k$ is identified with the subgroup of the symmetric group on $k$ elements corresponding to the cyclic permutations. As such, the map which takes an ordering on $S$ to its associated cycle induces a natural transformation of functors $Z^k\to C^k$.
\end{example}

A structure type $F$ is said to be of \emph{degree} $n$ if $F$ acts non-trivially only on sets of cardinality $n$. For example, $Z^n$ is of degree $n$. With any structure type $F$ we can then associate a countable sequence of structure types $(F_{(0)},F_{(1)},F_{(2)},\ldots)$, where $F_{(n)}$ is of degree $n$, so that if $S$ is a finite set with $n$ elements we have $F(S)=F_{(n)}(S)$ for all $n\in \mathbb{N}$. In such a case, we will often refer to the sequence $(F_{(0)},F_{(1)},F_{(2)},\ldots)$ as the \emph{expansion} of $F$, and $F_{(n)}$ as the \emph{degree n component} of $F$. Conversely, given a countable sequence of structure types $(F_{(0)},F_{(1)},F_{(2)},\ldots)$ with $F_{(n)}$ of degree $n$, one may associate a unique structure type $F$ such that if $S$ is a $n$-element set, then $F(S)=F_{(n)}(S)$. For example, the structure type associated with the sequence $(\bold{1},Z,Z^2,\ldots)$ is the structure of being a totally ordered set (irrespective of its cardinality), which we denote by $\mathscr{O}$. We then have
\[
\mathscr{O}(z)=\sum_{n=0}^{\infty}z^n=\frac{1}{1-z}.
\]
\section{The category of integral sets}\label{IS}
An \emph{integral set} is an ordered pair of finite sets $X=(X^+,X^-)$, and we refer to $X^+$ as the \emph{positive part} and $X^-$ as the \emph{negative part} of $X$ respectively. The union and intersection of integral sets are then defined component wise
\[
X\cup Y=(X^+\cup Y^+,X^-\cup Y^-),\quad \quad X\cap Y=(X^+\cap Y^+,X^-\cap Y^-).
\]
An intergral set $X$ is said to be a \emph{subset} of an integral set $Y$ and is written $X\subset Y$ if and only if 
\[
X^+\subset Y^+ \quad \text{and} \quad X^-\subset Y^-.
\]
A map $f:X\to Y$ between integral sets is defined to be a finite set map of the form 
\[
\tilde{f}: X^+\sqcup X^- \to Y^+\sqcup Y^-,
\] 
and $f$ is said to be \emph{coherent} if and only if 
\[
\tilde{f}(X^+)\subset Y^+, \quad \text{and} \quad \tilde{f}(X^-)\subset Y^-.
\]
The \emph{image} of an integral set map $f:X\to Y$ is the integral set given by
\[
\text{im}(f)=(\tilde{f}(X^+),\tilde{f}(X^-)),
\] 
so that $\text{im}(f)\subset Y$ if and only if $f$ is coherent. A coherent integral set map $f$ is said to be injective, surjective or bijective, if and only if the associated finite set map $\tilde{f}$ is, and two integral sets are said to be \emph{isomorphic} if and only if there exists a bijection between them. If an integral set map $f$ is not coherent but $\tilde{f}$ is injective, surjective, or bijective, we then say $f$ is \emph{weakly} injective, \emph{weakly} surjective, or \emph{weakly} bijective respectively. Given an integral set $X$ we let $-X=(X^-,X^+)$, and refer to it as \emph{negative}, or \emph{minus} $X$. It immediately follows that integral sets and maps between them form a category $\mathfrak{Int}$, which we refer to as the \emph{category of integral sets}. The category $\mathfrak{Fin}$ of finite sets and maps between them naturally embeds in $\mathfrak{Int}$ via the maps
\[
S\mapsto (S,\varnothing),\quad \text{or} \quad S\mapsto (\varnothing,S),
\]
the first of which we will refer to as the \emph{standard embedding} of $\mathfrak{Fin}$ in $\mathfrak{Int}$. Under the standard embedding $\mathfrak{Fin}\hookrightarrow \mathfrak{Int}$, a morphism $S\to T$ in $\mathfrak{Fin}$ naturally maps to a coherent map of integral sets. While isomorphism classes in $\mathfrak{Fin}$ are naturally indentified with the natural numbers $\mathbb{N}$, isomorphism classes in $\mathfrak{Int}$ may only be identified with the integers $\mathbb{Z}$ after a suitable equivalence relation, namely,
\begin{equation}\label{f1}
[X]\sim [Y] \quad \text{if and only if} \quad X^+\sqcup Y^-\cong Y^+\sqcup X^-,
\end{equation}
where $\sqcup$ denotes coproduct in $\mathfrak{Fin}$ (i.e., disjoint union), and $\cong$ denotes isomorphic equivalence (i.e., bijection). Under such an equivalence relation, the class of the integral set $(\varnothing,\{a,b,c,d,e\})$ is what we usually think of as $-5$, while the class of $(\{a,b,c\},\{a\})$ is what we usually think of as $2$. We record such observations via the following 
\begin{proposition}\label{t1}
Let $[\mathfrak{Int}]$ denote the set of isomorphism classes of objects in $\mathfrak{Int}$, and let $\sim$ denote the equivalence relation given by $(\ref{f1})$. Then the quotient $[\mathfrak{Int}]/\sim$ is in bijective correspondence with $\mathbb{Z}$. 
\begin{proof}
Since every member of $[\mathfrak{Int}]/\sim$ may be represented by an integral set of the form $(S,\varnothing)$ or $(\varnothing,T)$, the bijection $[\mathfrak{Int}]/\sim\to \mathbb{Z}$ is given by 
\[
(S,\varnothing)\mapsto \chi(S), \quad (\varnothing,T)\mapsto -\chi(T),
\]
where $\chi:\mathfrak{Fin}\to \mathbb{N}$ denotes the cardinality function for finite sets.
\end{proof}  
\end{proposition} 

An immediate corollary of Proposition~\ref{t1} is that the cardinality function on finite sets $\chi:\mathfrak{Fin}\to \mathbb{N}$ extends to a function $\chi:\mathfrak{Int}\to \mathbb{Z}$, which is compatible with the standard embedding $\mathfrak{Fin}\hookrightarrow \mathfrak{Int}$. The integers then parametrize isomorphism classes in $\mathfrak{Int}$ modulo the equivalence $\sim$, which we refer to as \emph{cardinal equivalence}. As the notion of cardinal equivalence involves an interaction between the positive and negative parts of an integral set, it is less rigid than isomorphic equivalence, which does not involve any interaction between the positive and negative parts of an integral set. In particular, non-isomorphic integral sets may share the same cardinality, such as $(\varnothing,\{a\})$ and $(\{a\},\{a,b\})$ (which both have cardinality $-1$), but as with the case of topological Euler characteristic, two objects with different cardinalities may never be isomorphic. However, two integral sets which are images of objects of $\mathfrak{Fin}$ under the standard embedding $\mathfrak{Fin}\hookrightarrow \mathfrak{Int}$ are isomorphic if and only if they have the same cardinality, so that the notion of cardinality in $\mathfrak{Fin}$ is faithfully preserved. In other words, at the level of objects we have the following commutative diagram.
\[
\xymatrix{
\mathfrak{Fin} \ar[d] \ar[r]  & \mathfrak{Int} \ar[d]\\
\mathbb{N} \ar[r] & \mathbb{Z}\\
}
\]
It also follows that for all integral sets $X$ we have
\[
\chi(-X)=-\chi(X).
\]

Coproducts exist in $\mathfrak{Int}$, which are given by 
\[
X\sqcup Y=(X^+\sqcup Y^+,X^-\sqcup Y^-),
\]
and it immediately follows that the associated inclusion  maps $\iota_X:X\to X\sqcup Y$ and $\iota_Y:Y\to X\sqcup Y$ are coherent. While a categorical product doesn't exist in $\mathfrak{Int}$, it is a distributive monoidal category with respect to the monoidal product given by 
\[
X\otimes Y=((X^+\times Y^+)\sqcup (X^-\times Y^-), (X^+\times Y^-)\sqcup (X^-\times Y^+)).
\]
The identity object with respect to this monoidal product is then $\mathbbm{1}=(\{\star\},\varnothing)$ (where $\{\star\}$ is a one-element set), and is unique up to canonical isomorphism. Given integral sets $X$, $Y$ and $Z$, the associator 
\[
\alpha_{X,Y,Z}:(X\otimes Y)\otimes Z\to X\otimes (Y\otimes Z)
\]
is established by identifying the positive and negative parts of its source and target via canonical isomorphisms in $\mathfrak{Fin}$. For example, the fact that $\mathfrak{Fin}$ is a distributive category allows one to identify the positive parts of the source and target of $\alpha_{X,Y,Z}$ via  the canonical isomorphism in $\mathfrak{Fin}$ between the sets
\[
(X^+\times Y^+\times Z^+) \sqcup (X^+\times Y^-\times Z^-) \sqcup (X^-\times Y^+\times Z^-) \sqcup (X^-\times Y^-\times Z^+),
\]
and 
\[
(X^+\times Y^+\times Z^+) \sqcup (X^-\times Y^-\times Z^+) \sqcup (X^+\times Y^-\times Z^-) \sqcup (X^-\times Y^+\times Z^-).
\]
Moreover, the fact that the pentagon diagram holds for the monoidal operation $\otimes$ follows from similar considerations, i.e., by using the fact that $\mathfrak{Fin}$ is a distributive category to identify positive and negative parts of sources and targets along the diagram via canonical isomorphisms in $\mathfrak{Fin}$.

With respect to cardinality, we then have
\[
\chi(X\sqcup Y)=\chi(X)+\chi(Y), \quad \text{and} \quad \chi(X\otimes Y)=\chi(X)\cdot \chi(Y),
\]
as expected. As the initial object in $\mathfrak{Int}$ is $(\varnothing,\varnothing)$, $-X\sqcup X$ is initial if and only if $X$ is, but from the formulas above we do have 
\[
\chi(-X\sqcup X)=0,
\]
so that $-X\sqcup X$ is cardinally equivalent to the initial object $(\varnothing,\varnothing)$.

Now even though the integers do not parametrize isomorphism classes in $\mathfrak{Int}$, there still is a precise sense in which equations among integers correspond to an explicit case of cardinal equivalence in $\mathfrak{Int}$. For example, the equation
\[
-3\cdot(5-3)+2=-1+(-3)
\]
corresponds to a cardinal equivalence of the form
\[
\left((\varnothing,\{x,y,z\})\times (\{\alpha,\beta,\gamma,\delta,\epsilon\},\{a,b,c\})\right)\sqcup (\{\star,\nu\},\varnothing)\sim (\varnothing,\star)\sqcup (\varnothing,\{x,y,z\}),
\]
where we recall that $\sim$ denotes cardinal equivalence.

\section{Integral structure types}\label{IST}
We now introduce the notion of an `integral structure type', so that we may categorify power series with possibly negative coefficients. An \emph{integral structure type} is a covariant functor 
\[
\Phi:\mathfrak{Fin}_0\to \mathfrak{Int},
\]
so that $\Phi$ acts on a finite set $S$ by the rule
\[
\Phi(S)=(\Phi^+(S),\Phi^-(S)),
\] 
where $\Phi^+$ and $\Phi^-$ are structure types. Integral structure types form the objects of a category $\mathfrak{Int}[Z]$, with morphisms given by natural transformations of functors. For example, a natural transformation $\nu:\Phi\to \Psi$ between two integral structure types associates with every bijection $f:S\to T$ of finite sets a commutative diagram of the form
\[
\xymatrix{
\Phi(S) \ar[d]_{\Phi(f)} \ar[r]^{\nu_S}  & \Psi(S) \ar[d]^{\Psi(f)} \\
\Phi(T) \ar[r]_{\nu_T} & \Psi(T) \\
}
\]
The monoidal operations of sum and product of structure types carry over to $\mathfrak{Int}[Z]$ as follows. The \emph{sum} of two integral structure types $\Phi=(\Phi^+,\Phi^-)$ and $\Psi=(\Psi^+,\Psi^-)$ is given by
\begin{equation}\label{f4}
\Phi+\Psi=(\Phi^++\Psi^+,\Phi^-+\Psi^-),
\end{equation}
and the \emph{product} of $\Phi$ and $\Psi$ is given by
\begin{equation}\label{f2}
\Phi\cdot \Psi=(\Phi^+\cdot \Psi^++\Phi^-\cdot \Psi^-,\Phi^+\cdot \Psi^-+\Phi^-\cdot \Psi^+).
\end{equation}
The identity with respect to sum is $(\bold{0},\bold{0})$, and the identity with respect to product is $(\bold{1},\bold{0})$ (recall that $\bold{0}$ is the structure type of being a set with an infinite number of elements, and $\bold{1}$ is the structure of being a set with no elements). It will be useful to note that structure types naturally embed in the category $\mathfrak{Int}[Z]$ via the map 
\begin{equation}\label{em}
F\mapsto (F,\bold{0}).
\end{equation}
The identity elements with respect to sum and product of structure types then map to the identity elements with respect to sum and product of integral structure types via the map (\ref{em}). We also note that if $F$ is a structure type, the integral structure type $(F,F)$ is \emph{not} isomorphic to $(\bold{0},\bold{0})$. 

The notions of degree and expansion readily extend to integral structure types, so that an expansion of an integral structure types is an ordered pair of expansions of its positive and negative parts. Given an integral structure type $\Phi=(\Phi^+,\Phi^-)$, functoriality along with the notion of cardinality for integral sets yields a well-defined notion of the (possibly negative) number of $\Phi$-structures one may put on an $n$-element set, which we denote by $\Phi_n\in \mathbb{Z}$. It immediately follows that 
\begin{equation}\label{f3}
\Phi_n=\Phi^+_n-\Phi^-_n,
\end{equation}
so that the exponential generating series $\Phi(z)$ associated with the sequence $\Phi_n$ is given by
\[
\Phi(z)=\Phi^+(z)-\Phi^-(z).
\]
As such, if $F$ is a structure type of the form $F=G+H$, then under the embedding of structure types in $\mathfrak{Int}[Z]$, we have 
\[
G(z)=F(z)-H(z),
\]
so that $G(z)$ is a generating series for $F$-structures which are not $H$-structures.

With the aforementioned prescriptions we then arrive at the following
\begin{theorem} 
Let $\Phi$ and $\Psi$ be integral structure types. Then
\[
(\Phi+\Psi)(z)=\Phi(z)+\Psi(z), \quad \text{and} \quad (\Phi\cdot \Psi)(z)=\Phi(z)\cdot \Psi(z).
\]
\begin{proof}
First, let $S_n$ denote an $n$-element set for all $n\in \mathbb{N}$. Then for structure types $F$ and $G$ we have  
\[
(F+G)_n=\chi((F+G)(S_n))=\chi(F(S_n)\sqcup G(S_n))=\chi(F(S_n))+\chi(G(S_n))=F_n+G_n,
\]
and since there are $\sum_{k=0}^n\binom{n}{k}$ ways to split up an $n$-element set into two disjoint subsets, from the definition of product of structure types we have
\[
(F\cdot G)_n=\sum_{k=0}^n\binom{n}{k}F_k\cdot G_{n-k}.
\]
Now let $\Phi=(\Phi^+,\Phi^-)$ and $\Psi=(\Psi^+,\Psi^-)$ be two integral structure types. Then by (\ref{f4}) and (\ref{f3}) we have
\begin{eqnarray*}
(\Phi+\Psi)_n&=&(\Phi+\Psi)^+_n-(\Phi+\Psi)^-_n \\
                     &=&(\Phi^++\Psi^+)_n-(\Phi^-+\Psi^-)_n \\
                     &=&\Phi^+_n+\Psi^+_n-(\Phi^-_n+\Psi^-_n) \\
                     &=&(\Phi^+_n-\Phi^-_n)+(\Psi^+_n-\Psi^-_n) \\
                     &=&\Phi_n+\Psi_n,
\end{eqnarray*}
so that
\begin{eqnarray*}
(\Phi+\Psi)(z)&=&\sum_{n=0}^{\infty}\frac{(\Phi+\Psi)_n}{n!}z^n \\
                          &=&\sum_{n=0}^{\infty}\frac{\Phi_n+\Psi_n}{n!}z^n \\
                          &=&\sum_{n=0}^{\infty}\frac{\Phi_n}{n!}z^n+\sum_{n=0}^{\infty}\frac{\Psi_n}{n!}z^n \\
                          &=&\Phi(z)+\Psi(z).
\end{eqnarray*}
Then for the product $\Phi\cdot \Psi$, combining our previous calculations with (\ref{f2}) yields 
\begin{eqnarray*}
(\Phi\cdot \Psi)_n&=&(\Phi\cdot \Psi)^+_n-(\Phi\cdot \Psi)^-_n \\
                           &=&(\Phi^+\cdot \Psi^++\Phi^-\cdot\Psi^-)_n-(\Phi^+\cdot \Psi^-+\Phi^-\cdot \Psi^+)_n \\
                           &=&(\Phi^+\cdot \Psi^+)_n+(\Phi^-\cdot \Psi^-)_n-((\Phi^+\cdot \Psi^-)_n+(\Phi^-\cdot \Psi^+)_n) \\
                           &=&\sum_{k=0}^n\binom{n}{k}(\Phi^+_k\cdot \Psi^+_{n-k}+\Phi^-_k\cdot \Psi^-_{n-k})-\sum_{k=0}^n\binom{n}{k}(\Phi^+_k\cdot \Psi^-_{n-k}+\Phi^-_k\cdot \Psi^+_{n-k}) \\
                           &=&\sum_{k=0}^n\binom{n}{k}(\Phi^+_k\cdot \Psi^+_{n-k}+\Phi^-_k\cdot \Psi^-_{n-k}-\Phi^+_k\cdot \Psi^-_{n-k}-\Phi^-_k\cdot \Psi^+_{n-k}) \\
                           &=&\sum_{k=0}^n\binom{n}{k}(\Phi^+_k-\Phi^-_k)\cdot (\Psi^+_{n-k}-\Psi^-_{n-k}) \\
                           &=&\sum_{k=0}^n\binom{n}{k}\Phi_k\cdot \Psi_{n-k}.
\end{eqnarray*}
And from the formula for the product of exponential generating series, namely
\[
\left(\sum_{n=0}^{\infty}\frac{a_n}{n!}z^n\right)\cdot \left(\sum_{n=0}^{\infty}\frac{b_n}{n!}z^n\right)=\sum_{n=0}^{\infty}\left(\frac{\sum_{k=0}^n\binom{n}{k}a_kb_{n-k}}{n!}\right)z^n,
\] 
we have 
\begin{eqnarray*}
(\Phi\cdot \Psi)(z)&=&\sum_{n=0}^{\infty}\frac{(\Phi\cdot \Psi)_n}{n!}z^n \\
                             &=&\sum_{n=0}^{\infty}\left(\frac{\sum_{k=0}^n\binom{n}{k}\Phi_k\cdot \Psi_{n-k}}{n!}\right)z^n \\
                             &=&\left(\sum_{n=0}^{\infty}\frac{\Phi_n}{n!}z^n\right)\cdot \left(\sum_{n=0}^{\infty}\frac{\Psi_n}{n!}z^n\right) \\
                             &=&\Phi(z)\cdot \Psi(z),
\end{eqnarray*}
as desired.
\end{proof}
\end{theorem}

In \cite{Joyal2}, Joyal defines the ring of \emph{virtual species}, which we refer to as \emph{virtual structure types}. Virtual structure types are simply ordered pairs of structure types with sum and product as given above for integral structure types, modulo the equivalence relation 
\[
\Phi=(\Phi^+,\Phi^-) \sim \Psi=(\Psi^+,\Psi^-) \quad \text{if and only if} \quad \Phi^++\Psi^-=\Phi^-+\Psi^+.
\] 
Integral structure types are then essentially virtual structure types without the aforementioned equivalence relation $\sim$, and thus contain more information than virtual structure types. Moreover, if $[\Phi]=[\Psi]$ as virtual structure types then for every finite set $S$ we have $\chi(\Phi(S))=\chi(\Psi(S))$, so that $\Phi(z)=\Psi(z)$. As such, the map which takes an integral structure type to its exponential generating series factors through the map which takes in integral structure type to its virtual structure type. When working with virtual structure types we will use the notation $F-G$ for the equivalence class of the pair $(F,G)$, since virtual structure types admit features of subtraction. In particular, if $F=G+H$ as structure types, then the virtual structure type $F-G$ is equivalent to $H-\bold{0}$, while the integral structure type $(F,G)$ is not isomorphic to $(H,\bold{0})$ for $G\neq \bold{0}$. 

\section{Operators on integral structure types}\label{OIST}
In the theory of generating functions certain operators on power series play a crucial role, such as the differentiation operator. As such, we first recall how the differentiation operator is lifted to structure types, and then give a general definition of operators on integral structure types. Our definition then naturally leads to the notion of the commutator of two operators on integral structure types, which will be useful for categorifying algebras of operators on power series.  

So let $F$ be a structure type. The \emph{derivative} of $F$, denoted $DF$, is a structure type which acts on a finite set $S$ by putting an $F$-structure on $S\sqcup \{\star\}$, where $\{\star\}$ denotes a set with a single element. The reason why such an operation on structure types is a combinatorial analogue of differentiation is evidenced by the following
\begin{example}\label{de1}
Let $F=Z^3$. Then
\begin{eqnarray*}
DZ^3(\{a,b\})&=&Z^3(\{a,b,\star\}) \\
                   &=&\{(a,b,\star),(a,\star,b),(b,a,\star),(b,\star,a),(\star,a,b),(\star,b,a)\} \\
                   &\cong& \{(b,\star),(\star,b)\}\sqcup \{(a,\star),(\star,a)\}\sqcup \{(a,b),(b,a)\} \\
                   &=&Z^2(\{b,\star\})\sqcup Z^2(\{a,\star\})\sqcup Z^2(\{a,b\}) \\
                   &\cong&3Z^2(\{a,b\}),
\end{eqnarray*}
thus $DZ^3\cong 3Z^2$. As such, we have
\[
(DZ^3)(z)=(3Z^2)(z)=3z^2=\frac{d}{dz}z^3.
\]
\end{example}

\begin{example}\label{ce1}
Let $C^4$ denote the structure of being a cycle of length $4$. A cycle will be denoted by an ordered tuple enclosed in rectangular brackets, such as $[w,x,y,z]$. Two ordered tuples then represent the same cycle if and only if they differ by a cyclic permutation. We then have
\begin{eqnarray*}
DC^4(\{x,y,z\})&=&C^4(\{x,y,z,\star\}) \\
                     &=&\{[x,y,z,\star],[x,z,y\star],[y,x,z,\star],[y,z,x,\star],[z,x,y,\star],[z,y,x,\star]\} \\
                     &\cong& \{(x,y,z),(x,z,y),(y,x,z),(y,z,x),(z,x,y),(z,y,x)\} \\
                     &=&Z^3(\{x,y,z\}),
\end{eqnarray*}
thus $DC^4\cong Z^3$.
\end{example}

More generally, one may show that for any structure type $F$ we have
\[
DF(z)=\frac{d}{dz}F(z),
\]
so that indeed the operator $D$ is a categorification of differentiation of power series. The differentiation operator may then be viewed as an endofunctor
\[
D:\mathfrak{Fin}[Z]\to \mathfrak{Fin}[Z],
\]
such that
\[
D(F+G)=D(F)+D(G) \quad \text{and} \quad D(F\cdot G)=D(F)\cdot G+F\cdot D(G).
\]
Moreover, if $F\to G$ is a natural transformation of structure types, and $S\to T$ is a bijection of finite sets, then the morphism $D(F\to G)$ corresponds to a morphism of diagrams of the form
\[
\xymatrix{
F(S) \ar[d] \ar[r]  & G(S) \ar[d] \\
F(T) \ar[r] & G(T) \\
}
\quad
\implies
\quad
\xymatrix{
F(S\sqcup \{\star\}) \ar[d] \ar[r]  & G(S\sqcup \{\star\}) \ar[d] \\
F(T\sqcup \{\star\}) \ar[r] & G(T\sqcup \{\star\}) \\
}
\]
Taking the differentiation operator as a representative example, we define an \emph{operator} on a structure type to be an endo-functor of the form
\begin{equation}\label{odst}
\mathscr{A}:\mathfrak{Fin}[Z]\to \mathfrak{Fin}[Z],
\end{equation}
such that $\mathscr{A}(F+G)=\mathscr{A}(F)+\mathscr{A}(G)$. Given two operators on structure types $\mathscr{A}$ and $\mathscr{B}$ we define their sum as
\[
(\mathscr{A}+\mathscr{B})F=\mathscr{A}F+\mathscr{B}F
\]
their product as
\[
(\mathscr{A}\cdot \mathscr{B})F=\mathscr{A}F\cdot \mathscr{B}F,
\]
and their composition as 
\[
(\mathscr{A}\circ \mathscr{B})F=\mathscr{A}(\mathscr{B}F),
\]
where $F$ denotes an arbitrary structure type.

We then define operators on integral structure types by replacing $\mathfrak{Fin}[Z]$ by $\mathfrak{Int}[Z]$ in the definition of an operator on structure types (\ref{odst}). A general operator on integral structure types will then be of the form $\mathscr{A}=(\mathscr{A}^+,\mathscr{A}^-)$ with $\mathscr{A^{\pm}}$ operators on structure types, so that
\[
\mathscr{A}\Phi(S)=(\mathscr{A}^+\Phi^+(S),\mathscr{A}^-\Phi^-(S)).
\]
Since the differentiation operator $D$ essentially corresponds to an operation on the \emph{sets} upon which structure types act (rather than the structure types themselves), there is only one sensible way to extend differentiation to integral structure types, namely, by setting
\[
D\Phi(S)=\Phi(S\sqcup \{\star\})=(\Phi^+(S\sqcup \{\star\}),\Phi^-(S\sqcup \{\star\}))=(D\Phi^+(S),D\Phi^-(S)).
\] 
Certainly any operator $\mathscr{A}$ on structure types extends to an operator on integral structure types in the same way as the differentiation operator $D$, namely, as $(\mathscr{A},\mathscr{A})$. Operators on integral structure types of the form $(\mathscr{A},\mathscr{A})$ will then be referred to as \emph{pure}, and operators of the form $(\mathscr{A}^+,\mathscr{A}^-)$ with $\mathscr{A}^+\neq \mathscr{A}^-$ will be referred to as \emph{mixed}. If $\mathscr{A}$ is an operator on structure types its associated pure operator on integral structure types will often be denoted simply by $\mathscr{A}$ as well. We may associate with two pure operators $\mathscr{A}$ and $\mathscr{B}$ on integral structure types the mixed operator $[\mathscr{A},\mathscr{B}]=(\mathscr{A}\circ \mathscr{B},\mathscr{B}\circ \mathscr{A})$, which we will refer to as the \emph{commutator} of $\mathscr{A}$ and $\mathscr{B}$. We then have
\[
[\mathscr{A},\mathscr{B}]F(z)=(\mathscr{A}\circ \mathscr{B})F(z)-(\mathscr{B}\circ \mathscr{A})F(z).
\]
If the a priori mixed operator $[\mathscr{A},\mathscr{B}]$ is in fact pure, we say $\mathscr{A}$ and $\mathscr{B}$ \emph{commute}. In working with commutators, in some sense it seems more natural to work with virtual structure types as opposed to integral structure types, as evidenced by the following 
\begin{example}
Let $M_Z$ denote the operator on structure types given by
\[
M_ZF(S)=\bigsqcup_{s\in S}F(S\setminus s). 
\]   
From the definition of product of structure types we have
\[
M_ZF=Z\cdot F,
\] 
since putting an $M_ZF$ structure on $S$ corresponds to partitioning $S$ into a singleton and its complement, putting a $Z$ structure on the singleton and then putting an $F$-structure on its complement. One then has
\begin{equation}\label{sti}
(D\circ M_Z)F\cong F+(M_Z\circ D)F,
\end{equation}
which is a lift to the level of structure types of the power series identity
\[
\left[\frac{d}{dz},m_z\right]f(z)=\left(\frac{d}{dz}\circ m_z-m_z\circ \frac{d}{dz}\right)f(z)=f(z),
\]
where $m_z$ denotes the operator multiplication by $z$. Then viewing $D$ and $M_Z$ as pure operators on integral structure types, the ismorphism (\ref{sti}) yields
\[
[D,M_Z]F\cong (F+(M_Z\circ D)F,(M_Z\circ D)F)\sim F-\bold{0},
\] 
where $\sim$ denotes the equivalence relation defining virtual structure types. The commutator $[D,M_Z]$ at the level of integral structure types then carries the extra term $M_Z\circ D$ in both its positive and negative components, while the associated virtual structure type `cancels' the term $M_Z\circ D$.
\end{example}

We now set out to define an operator on a sub-class of structure types which we refer to as `regular'. The extension of this operator to a pure operator on integral structure types will be useful in \S\ref{CCC}. A structure type $F$ is said to be \emph{regular} if every element of $F(S)$ has no non-trivial symmetries for all finite sets $S$. More precisely, let $F$ be a structure type, $S$ be an $n$-element set, and let $\mathcal{S}_n$ denote the symmetric group on $n$ elements. Then there exists an action of $\mathcal{S}_n$ on $F(S)$ by taking an $F$-structure $x\in F(S)$ to the $F$-structure $sx\in F(S)$ obtained by permuting the underlying elements of $S$ according to the group element $s\in \mathcal{S}_n$. We then define the \emph{automorphism} group of $x\in F(S)$ as
\[
\text{Aut}(x)=\{s\in \mathcal{S}_n \hspace{1mm}| \hspace{1mm} sx=x\}\subset \mathcal{S}_n.
\]
An $F$-structure $x\in F(S)$ is then said to have no non-trivial symmetries when its automorphism group consists only of the identity permutation. A structure type $F$ is then regular if for every finite set $S$ we have $\chi(\text{Aut}(x))=1$ for all $x\in F(S)$. Crucial for what follows will be the fact that if $F_{(n)}$ is a regular structure type of degree $n$, and $G$ is a subgroup of $\mathcal{S}_n$, then $F_{(n)}(S)/G$ is isomorphic to $H_{(n)}(S)$ for some degree-$n$ structure type $H_{(n)}$ \cite{BergeronEA}.   

So let $F$ be a regular structure type, with expansion
\[
F=(F_{(0)},F_{(1)},\ldots, F_{(n)},\ldots).
\]
Certainly each $F_{(n)}$ is regular for all $n\in \mathbb{N}$. Now let $S$ be a set with $n-1$ elements. We then denote by $\mathscr{X}$ the operator given by  
\begin{equation}\label{dno}
\mathscr{X}F(S)=F_{(n)}(S\sqcup \{\star\})/\mathbb{Z}_n.
\end{equation}
As previously mentioned, the regularity of $F_{(n)}$ implies that $F_{(n)}(S\sqcup \{\star\})/\mathbb{Z}_n$ is isomorphic to $H_{(n)}(S\sqcup \{\star\})$ for some degree-$n$ structure type $H_{(n)}$, so that $\mathscr{X}F(S)=H_{(n)}(S\sqcup \{\star\})$ (and is thus a well-defined structure type). In spite of its abstract definition, $\mathscr{X}$ admits a very simple description when acting on linear orderings. In particular, let $S$ denote a set with $n-1$ elements and denote by $C^n$ the structure type of being a cycle of length $n$. Then
\begin{equation}\label{no}
\mathscr{X}Z^n(S)=Z^n(S\sqcup \{\star\})/\mathbb{Z}_n\cong C^n(S\sqcup \star)\cong Z^{n-1}(S),
\end{equation}
thus $\mathscr{X}Z^n\cong Z^{n-1}$ (so that $H_{(n)}=C^n$ in this case).

\section{Integral stuff types}\label{STF}
Baez and Dolan introduced a way to repackage the information contained in a structure type which is amenable to further generalization, and is more aptly suited to categorify evaluation of power series at a number \cite{BaezDolan}.  In particular, we first recall that a \emph{groupoid} is a category for which all the morphisms are in fact \emph{iso}morphisms. Now given a structure type $F$, for every finite set $S$ we may think of the elements of $F(S)$ as lying in a fiber over $S$, and the collection of all such fibers forms the objects of a groupoid $\bold{X}_F$ lying over $\mathfrak{Fin}_0$. For every bijection $S\to T$ in $\mathfrak{Fin}_0$, there exist isomorphisms between $F$-structures in the fibers over $S$ and $T$ in $\bold{X}_F$ corresponding to relabeling the elements of $S$ in $F$-structures on $S$ by elements of $T$ according to the bijection $S\to T$. Such isomorphisms are precisely the morphisms in the groupoid $\bold{X}_F$. We may then think of a structure type $F$ as a functor between groupoids
\[
F:\bold{X}_F\to \mathfrak{Fin}_0,
\] 
where an $F$-structure is taken to its underlying set, and ismorphisms of $F$-structures are taken to bijections between the underlying sets of the $F$-structures.

With the above construction in mind, we define a \emph{stuff type} $F$ to be a functor between groupoids
\[
F:\bold{X}\to \mathfrak{Fin}_0,
\]
such that for every bijection $S\to T$ of finite sets there exists a commutative diagram
\[
\xymatrix{
F^{-1}(S) \ar[d] \ar[r]  & F^{-1}(T) \ar[d] \\
S \ar[r] & T \\
}
\] 
Not all stuff types are structure types. In particular, structure types are precisely the stuff types $F:\bold{X}\to \mathfrak{Fin}_0$ for which the functor $F$ is \emph{faithful}, which means that given two objects $A$ and $B$ of $\bold{X}$ the associated map
\[
\text{Hom}(A,B)\to \text{Hom}(F(A),F(B))
\]
is injective. Faithfulness of structure types then follows from the fact that morphisms between objects in a structure type $F:\bold{X}\to \mathfrak{Fin}_0$ always correspond to a relabeling of the elements of its underlying set, so that there is a one to one correspondence between morphisms in a structure type $\bold{X}$ and morphisms in $\mathfrak{Fin}_0$. To define a generating series associated with an arbitrary stuff type which generalizes the generating series of a structure type, we need to recall the concept of `groupoid cardinality', which was first introduced in \cite{BaezDolan}. 

Given a groupoid $\mathcal{G}$, denote the set of isomorphisms classes of objects in $\mathcal{G}$ by $[\mathcal{G}]$. The \emph{groupoid cardinality} of $\mathcal{G}$ is then given by
\[
\chi(\mathcal{G})=\sum_{[x]\in [\mathcal{G}]}\frac{1}{\chi(\text{Aut}(x))},
\] 
where $\text{Aut}(x)$ denotes the set of automorphisms of the object $x$ in $\mathcal{G}$. When the above sum diverges we set $\chi(\mathcal{G})=\infty$, and groupoids for which the above sum converges are referred to as \emph{tame}. For all groupoids $\mathcal{G}$ and $\mathcal{H}$ we then have
\[
\chi(\mathcal{G}+\mathcal{H})=\chi(\mathcal{G})+\chi(\mathcal{H}), \quad \text{and} \quad \chi(\mathcal{G}\times \mathcal{H})=\chi(\mathcal{G})\cdot \chi(\mathcal{H}),
\]
where $\mathcal{G}+\mathcal{H}$ denotes direct sum and $\mathcal{G}\times \mathcal{H}$ denotes the product of groupoids.

If $F:\bold{X}\to \mathfrak{Fin}_0$ is a stuff type, the groupoid $\bold{X}$ is then a direct sum of groupoids
\begin{equation}\label{est}
\bold{X}=\sum_{n=0}^{\infty}\bold{X}_n,
\end{equation}
where $\bold{X}_n$ is the union of the fibers over all $n$-element sets. Equation (\ref{est}) is then referred to as the \emph{expansion} of the stuff type $\bold{X}$. The stuff type $\bold{X}_n\to \mathfrak{Fin}_0$ is then said to be the \emph{degree}-$n$ component of $\bold{X}$. In particular, if $\bold{X}=\bold{X}_n$ (so that $\bold{X}_k$ is the empty groupoid for $k\neq n$), then $\bold{X}$ is said to be of degree $n$. We refer to a stuff type $\bold{X}$ as \emph{relatively tame} if $\bold{X}_n$ is tame for all $n\in \mathbb{N}$. With every relatively tame stuff type $F:\bold{X}\to \mathfrak{Fin}_0$ we then associate the formal power series
\begin{equation}\label{gsst}
\bold{X}(z)=\sum_{n=0}^{\infty}\chi(\bold{X}_n)z^n,
\end{equation}
which we refer to as the \emph{generating series} for the stuff type $\bold{X}$. Note that stuff types which are in fact structure types are always relatively tame, and that the groupoid cardinality of a stuff type $\bold{X}$ coincides with $\bold{X}(1)$.

To make the connection with generating series for structure types, suppose $F:\bold{X}\to \mathfrak{Fin}_0$ is a stuff type which is in fact a structure type. The fact that its generating series as a structure type coincides with its generating series as a stuff type as given by (\ref{gsst})  follows from the fact that the degree-$n$ component $\bold{X}_n$ of $\bold{X}$ is in fact isomorphic to the action groupoid $F(S)//\mathcal{S}_n$, where $S$ is an $n$-element set, $F(S)$ is its image when viewing $F$ as a structure type, and $\mathcal{S}_n$ denotes the symmetric group on $n$ elements. The action groupoid $F(S)//\mathcal{S}_n$ has the elements of $F(S)$ as its objects, with an arrow between two objects precisely when they differ by a relabeling under the action of $\mathcal{S}_n$ on its underlying set. Notice that these morphisms exist in $\bold{X}$ as well, since the elements of $F(S)$ are also objects of $\bold{X}_n$, and a bijection $S\to S$ induces morphisms between the objects $F(S)$ of $\bold{X}_n$ which differ by a permutation of its underlying set. The connected components corresponding to the isomorphism classes of $\bold{X}_n$ are then in one-to one correspondence with isomorphism classes in the action groupoid $F(S)//\mathcal{S}_n$. One may then show that the groupoid cardinality of the action groupoid $F(S)//\mathcal{S}_n$ is $\chi(F(S))/\chi(\mathcal{S}_n)$, thus
\[
\chi(\bold{X}_n)=\chi(F(S)//\mathcal{S}_n)=\frac{F_n}{n!},
\] 
where $F_n=\chi(F(S))$ is as denoted in the exponential generating series for $F$ as a structure type.

Given two stuff types $\bold{X}\to \mathfrak{Fin}_0$ and $\bold{Y}\to \mathfrak{Fin}_0$ we define their sum $(\bold{X}+\bold{Y})\to \mathfrak{Fin}_0$ via the expansion 
\[
\bold{X}+\bold{Y}=\sum_{n=0}^{\infty}(\bold{X}_n+\bold{Y}_n),
\]
where $\bold{X}_n+\bold{Y}_n$ denotes the direct sum of the degree-$n$ components of  $\bold{X}$ and $\bold{Y}$. The product of $\bold{X}\to \mathfrak{Fin}_0$ and $\bold{Y}\to \mathfrak{Fin}_0$ is then given by
\[
\bold{X}\cdot \bold{Y}=\sum_{n=0}^{\infty}\left(\sum_{i+j=n}\bold{X}_i\times \bold{Y}_j\right),
\]
where $\bold{X}_i\times \bold{Y}_j$ denotes the product of $\bold{X}_i$ and $\bold{Y}_j$ in the category of groupoids. As for the associated map $\bold{X}\cdot \bold{Y}\to \mathfrak{Fin}_0$, note that composing the the first and second projections of the product $\bold{X}_i\times \bold{Y}_j$ with the maps $\bold{X}_i\to \mathfrak{Fin}_0$ and $\bold{Y}_j\to \mathfrak{Fin}_0$  yield maps onto finite sets of cardinality $i$ and $j$ respectively (with $i+j=n$), so that taking the disjoint union of the image of these maps yields an $n$-element set. It is then straightforward to show that with such prescriptions we recover the notions of sum and product of structure types. As groupoid cardinality is additive with respect to direct sum and multiplicative with respect to products, if $\bold{X}$ and $\bold{Y}$ are both relatively tame we then have
\[
(\bold{X}+\bold{Y})(z)=\bold{X}(z)+\bold{Y}(z) \quad \text{and} \quad (\bold{X}\cdot \bold{Y})(z)=\bold{X}(z)\cdot \bold{Y}(z).
\]
 
We now interpret integral structure types in the language of stuff types, which leads to the notion of an `integral stuff type'. For this, we define an \emph{integral groupoid} to be an ordered pair of groupoids $\mathcal{G}=(\mathcal{G}^+,\mathcal{G}^-)$, and we refer to $\mathcal{G}^+$ as the \emph{positive part} and $\mathcal{G}^-$ as the \emph{negative part} of $\mathcal{G}$ respectively. Given two integral groupoids $\mathcal{G}=(\mathcal{G}^+,\mathcal{G}^-)$ and $\mathcal{H}=(\mathcal{H}^+,\mathcal{H}^-)$, we define their direct sum as
\[
\mathcal{G}+\mathcal{H}=(\mathcal{G}^++\mathcal{H}^+,\mathcal{G}^-+\mathcal{H}^-),
\]
and their product as
\[
\mathcal{G}\times \mathcal{H}=((\mathcal{G}^+\times \mathcal{H}^+)+(\mathcal{G}^-\times \mathcal{H}^-),(\mathcal{G}^+\times \mathcal{H}^-)+(\mathcal{G}^-\times \mathcal{H}^+)).
\]
We then define the cardinality of an integral groupoid $\mathcal{G}=(\mathcal{G}^+,\mathcal{G}^-)$ to be
\[
\chi(\mathcal{G})=\chi(\mathcal{G}^+)-\chi(\mathcal{G}^-),
\]
and $\mathcal{G}$ is said to be \emph{relatively tame} if both its positive and negative parts are. 

An \emph{integral stuff type} is then given by a covariant functor
\[
\bold{\Phi}\to \mathfrak{Fin}_0,
\]
with $\bold{\Phi}=(\bold{\Phi}^+,\bold{\Phi}^-)$ an integral groupoid, and $\bold{\Phi}^{\pm}\to \mathfrak{Fin}_0$ both being stuff types. If $\bold{\Phi}=(\bold{\Phi}^+,\bold{\Phi}^-)$ is relatively tame its associated generating series is given by
\[
\bold{\Phi}(z)=\bold{\Phi}^+(z)-\bold{\Phi}^-(z).
\] 
The expansion of an integral stuff type $\bold{\Phi}=(\bold{\Phi}^+,\bold{\Phi}^-)$ is given by
\[
\bold{\Phi}=\sum_{n=0}^{\infty}\bold{\Phi}_n,
\]
where $\bold{\Phi}_n=(\bold{\Phi}^+_n,\bold{\Phi}^-_n)$. With such prescriptions, the sum and product of two integral stuff types are defined in exactly the same manner as stuff types, from which it follows
\[
(\bold{\Phi}+\bold{\Psi})(z)=\bold{\Phi}(z)+\bold{\Psi}(z),\quad \text{and} \quad (\bold{\Phi}\cdot \bold{\Psi})(z)=\bold{\Phi}(z)\cdot \bold{\Psi}(z).
\]

\section{A combinatorial description for Chern classes of projective hypersurfaces}\label{CCC}
We now use operators on integral structure types to give a combinatorial description for Chern classes of projective space and all of its smooth hypersurfaces. So let $\mathbb{P}^n$ denote projective $n$-space over $\C$, and let $X$ be a smooth complex projective hypersurface, i.e., a subset of $\mathbb{P}^n$ corresponding to the zero locus of a homogeneous polynomial $F(x_0,...,x_n)\in \C[x_0,...,x_n]$ such that the system of equations
\[
\frac{\partial F}{\partial x_i}=0, \quad \quad i=0,...,n
\] 
has no solutions. As a topological space $X$ is of dimension $2(n-1)$, and is a complex manifold of complex dimension $n-1$. A fundamental invariant of the embedding $X\hookrightarrow \mathbb{P}^n$ is the \emph{degree} of $X$, which is simply the degree of the homogeneous polynomial defining $X$. Geometrically, the degree of $X$ is the number of points of intersection of $X$ with $n-1$ hyperplanes in general position, i.e.,
\[
\text{deg}(X)=X\cap H_1\cap \cdots \cap H_{n-1},
\]
where $H_i$ denotes a general hyperplane in $\mathbb{P}^n$ (i.e., a hypersurface of degree 1). Many fundamental invariants of $X$ are determined by its degree, such as its Hodge structure and Chern classes. In particular, its topological Euler characteristic $\chi(X)$ is a polynomial in its degree \cite{IT}, given by
\begin{equation}\label{ecf}
\chi(X)=-\sum_{k=0}^{n-1}\binom{n+1}{k}(-d)^{n-k},
\end{equation}
where $d$ denotes the degree of $X$. The appearance of binomial coefficients in the formula for the Euler characteristic of $X$ along with the fact that the Euler characteristic is the unique measure on toplogical spaces which extends the cardinality function of finite sets hints at the fact that the Euler characteristic is in fact combinatorial in nature. In this direction, we invoke the use of operators on integral structure types to give a combinatorial description of the Euler characteristic of $X$ and its relation to its \emph{total} Chern class. 

For those not familiar with Chern classes in algebraic geometry, we first recall that one interpretation of the Euler characteristic of a variety is that it counts the zeros of a non-trivial holomorphic vector field on it (with either positive or negative multiplicities). For example, a flow of charge on the Riemann sphere must admit two poles where the charge vanishes, which comes from the fact that the Euler characteristic of the Riemann sphere is 2. The Euler characteristic is then \emph{zero-dimensional information}. To generalize this to higher-dimensional information, one may consider $k$ general vector fields on a variety and determine the locus for which the $k$ vector fields are necessarily linearly dependent (over $\mathbb{C}$). In general, the locus where $k$ vector fields are linearly dependent on a variety is of complex dimension $k-1$, so if $m$ is the dimension of a variety $V$, we have a sequence
\begin{equation}\label{cce}
c_0(V)+c_1(V)+\cdots +c_m(V),
\end{equation}
where $c_i(V)$ denotes the locus in $V$ where $m+1-i$ vector fields are necessarily linearly dependent. Taking such loci up to (co)homological equivalence, $c_i(V)$ is then referred to as the $i$th \emph{Chern class} of $V$, and the homological sum \eqref{cce} is then referred to as the \emph{total Chern class} of $V$, which we denote by $c(V)$ (for an introductory (and more rigorous) account of Chern classes in algebraic geometry we recommend \cite{Teitler}). And since the Euler characteristic of $V$ corresponds to $c_m(V)$, the Euler characteristic is often referred to as the \emph{top Chern class} of $V$ (though technically, the top Chern class $c_m(V)$ is the Euler characteristic of $V$ times the homological class of a point). In what follows, we construct a structure type on sets of cardinality less than or equal to $n+1$ whose generating series may be viewed as the total Chern class of $\mathbb{P}^n$, and then show how the total Chern classes of \emph{all} smooth hypersurfaces in $\mathbb{P}^n$ may be recovered from this structure type. But first, we recall certain combinatorial aspects of $\mathbb{P}^n$ which will be crucial for our construction.  

We first describe how the total Chern class of $\mathbb{P}^n$ may be viewed as a generating function for the number of linear subspaces in a `skeletal' version of $\mathbb{P}^n$, which may be modeled by an $n$-simplex. In particular, one may consider projective $n$-space over a field $\mathbb{F}_q$ with $q=p^k$ elements (with $p$ a prime number), whose points correspond to one-dimensional subspaces of an $n+1$-dimensional vector space over $\mathbb{F}_q$, which we denote by $\mathbb{F}_q^{n+1}$. The number $N_q(n+1,k)$ of $k$-dimensional linear subspaces of $\mathbb{F}_q^{n+1}$ then corresponds to the number of embedded $\mathbb{P}^{k-1}$s in $\mathbb{P}^n$ over $\mathbb{F}_q$, which is given by 
\[
N_q(n+1,k)=\binom{n+1}{k}_q=\frac{(1-q^{n+1})(1-q^n)\cdots (1-q^{n+2-k})}{(1-q)(1-q^2)\cdots (1-q^k)},
\]
where $\binom{n+1}{k}_q$ is the $q$-analogue of the binomial coefficient $\binom{n+1}{k}$. Setting $q=1$ then corresponds to counting the number of $(k-1)$-dimensional linear subspaces of projective $n$-space over $\mathbb{F}_1$, i.e., the `field' with one element. We then refer to $\mathbb{P}^n$ over $\mathbb{F}_1$ as the \emph{skeletal} projective $n$-space, denoted $\mathfrak{P}^n$, which has
\[
\binom{n+1}{1} \text{ points ($\mathfrak{P}^0$s),} \quad \binom{n+1}{2} \text{ lines ($\mathfrak{P}^1$s),} \ldots , \quad \text{and} \quad \binom{n+1}{n} \text{ hyperplanes ($\mathfrak{P}^{n-1}$s).}  
\]
As such, the geometry of skeletal projective space $\mathfrak{P}^n$ is completely captured by the combinatorics of  an $n$-simplex, which has
\[
\binom{n+1}{1} \text{ vertices,} \quad \binom{n+1}{2} \text{ edges,} \ldots,  \quad \text{and} \quad \binom{n+1}{n} \text{ $(n-1)$-dimensional faces,} 
\]
so that each $k$-dimensional face of the $n$-simplex corresponds to an embedded $\mathfrak{P}^k$ inside of $\mathfrak{P}^n$ (we give an illustration of $\mathfrak{P}^3$ in Figure~\ref{M1}). We point out that the number of points in $\mathfrak{P}^n$ coincides with the Euler characteristic of $\mathbb{P}^n$ over $\C$,  namely, $n+1$.
  
Now denote the integral cohomology ring of complex projective $n$-space by $A_*\mathbb{P}^n$, which one may identify with its Chow group of algebraic cycles modulo rational equivalence \cite{IT}. If we denote the class of a hyperplane in $\mathbb{P}^n$ by $H$, we have
\[
A_*\mathbb{P}^n\cong \mathbb{Z}[H]/(H^{n+1}),
\] 
so that cohomology classes in $\mathbb{P}^n$ are just polynomials in $H$ with integer coefficients. It is well-known that the total Chern class of $\mathbb{P}^n$ is given by
\[
c(\mathbb{P}^n)=(1+H)^{n+1}(\text{mod}\hspace{1mm}H^{n+1})=\sum_{i=0}^n\binom{n+1}{i}H^i,
\]
from which one can see $\chi(\mathbb{P}^n)=n+1$ (since $H^i$ is a class of dimension $n-i$, $\chi(\mathbb{P}^n)$ coincides with the coefficient of $H^n$). By the preceding discussion, we may view the total Chern class of $\mathbb{P}^n$ as a generating function, where the coefficient of $H^i$ counts the number of embedded $\mathfrak{P}^{(n-i)}$s in $\mathfrak{P}^n$, or rather, the number of $(n-i)$-dimensional faces in the $n$-simplex. We now use the combinatorial structure of skeletal projective space to show how the total Chern class of $\mathbb{P}^n$ may be viewed as the generating series of a structure type. We first give a concrete example in the case of $\mathfrak{P}^3$, which immediately generalizes to $\mathfrak{P}^n$ for arbitrary $n$.

\begin{figure}
\centering
\begin{tikzpicture}

\tikzset{vertex/.style = {shape=circle,draw,minimum size=1.5em}}
\tikzset{edge/.style = {->,> = latex}}
\node[vertex] (a) at  (0,0) {$x_0$};
\node[vertex] (b) at  (2,0) {$x_1$};
\node[vertex] (c) at  (0,2) {$x_2$};
\node[vertex] (d) at  (-1.5,-1.5) {$x_3$};

\draw (a) -- (b);
\draw (b) -- (c);
\draw (a) -- (d);
\draw (c) -- (d);
\draw (a) -- (c);
\draw (d) -- (b);

\end{tikzpicture}

\caption{Skeletal projective 3-space $\mathfrak{P}^3$} \label{M1}
\end{figure}
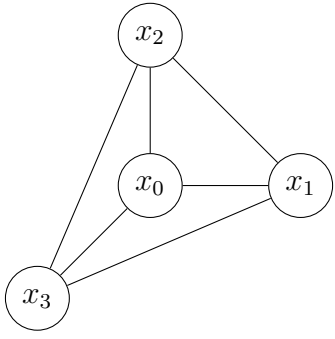

So let $\mathfrak{C}^3$ be the structure type of being an embedded linear subspace of $\mathfrak{P}^3$ equipped with a linear ordering on it (by linear subspace we mean that it is isomorphic to $\mathfrak{P}^k$ for some $k\leq 3$). For example, $\mathfrak{C}^3(\{\heartsuit,\clubsuit,\spadesuit\})$ consists of all skeletal $\mathfrak{P}^3$s with three distinct vertices labeled by $\{\heartsuit,\clubsuit,\spadesuit\}$ together with an ordering on $\{\heartsuit,\clubsuit,\spadesuit\}$, which may be represented by replacing two undirected edges in the simplex corresponding to $\mathfrak{P}^3$ by directed edges corresponding to the ordering of $\{\heartsuit,\clubsuit,\spadesuit\}$ (we sketch an element of $\mathfrak{C}^3(\{\heartsuit,\clubsuit,\spadesuit\})$ in Figure~\ref{M2}). As such, for a finite set $S$, $\mathfrak{C}^3(S)$ may be identified with the set of injective maps of the form $S\to \{x_0,x_1,x_2,x_3\}$. The structure type $\mathfrak{C}^3$ then admits the expansion
\[
\mathfrak{C}^3=(\bold{1},\mathfrak{C}^3_{(1)},\mathfrak{C}^3_{(2)}, \mathfrak{C}^3_{(3)}, \mathfrak{C}^3_{(4)},\bold{0}, \bold{0},\ldots),
\] 
where $\mathfrak{C}^3_{(k)}$ is the degree $k$ structure type of being an embedded $\mathfrak{P}^{k-1}$ in $\mathfrak{P}^3$ with a linear ordering for $k=1,2,3,4$. If $S$ is a $k$-element set we then have 
\[
\chi(\mathfrak{C}^3_{(k)}(S))=\frac{4!}{(4-k)!},
\]
thus the generating series for $\mathfrak{C}^3$ is given by
\[
\mathfrak{C}^3(z)=\sum_{k=0}^4\frac{4!}{k!(4-k)!}z^k=\sum_{k=0}^4\binom{4}{k}z^k=(1+z)^4,
\]
from which it follows  
\[
c(\mathbb{P}^3)=\mathfrak{C}^3(H)(\text{mod}\hspace{1mm}H^{4})\in A_*\mathbb{P}^n.
\]
We note that the elements of $\mathfrak{C}^3_{(k)}(S)$ are in bijective correspondence with length-$k$ flags of linear subspaces of $\mathfrak{P}^3$. For example, if we choose the linear ordering $(z,y,x)$ on the set $\{x,y,z\}$, then the embedding of $\{x,y,z\}$ into a $\mathfrak{P}^2\subset \mathfrak{P}^3$ given by $z\mapsto x_0$, $y\mapsto x_1$, $x\mapsto x_3$ then induces the length-3 flag of linear subspaces $\mathfrak{P}^0\subset \mathfrak{P}^1\subset \mathfrak{P}^2$ in $\mathfrak{P}^3$, where
\[
\mathfrak{P}^0=\{x_0\}, \quad \mathfrak{P}^1=\{x_0,x_1\}, \quad \text{and} \quad \mathfrak{P}^2=\{x_0,x_1,x_3\}.
\]

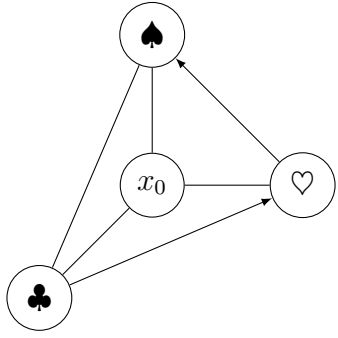
\begin{figure}
\centering
\begin{tikzpicture}

\tikzset{vertex/.style = {shape=circle,draw,minimum size=1.5em}}
\tikzset{edge/.style = {->,> = latex}}
\node[vertex] (a) at  (0,0) {$x_0$};
\node[vertex] (b) at  (2,0) {$\heartsuit$};
\node[vertex] (c) at  (0,2) {$\spadesuit$};
\node[vertex] (d) at  (-1.5,-1.5) {$\clubsuit$};

\draw (a) -- (b);
\draw[edge] (b) to (c);
\draw (a) -- (d);
\draw (c) -- (d);
\draw (a) -- (c);
\draw[edge] (d) to (b);

\end{tikzpicture}

\caption{An element of $\mathfrak{C}^3(\{\heartsuit,\clubsuit,\spadesuit\})$} \label{M2}
\end{figure}

For general $n$ the construction is the same as for $n=3$, that is, we define $\mathfrak{C}^n$ to be the structure of being an embedded linear subspace of $\mathfrak{P}^n$ equipped with a linear ordering, which again may be indentitied with the set of flags of linear subspaces of $\mathfrak{P}^n$. As in the $n=3$ case, for general $n$ the structure type $\mathfrak{C}^n$ admits the expansion
\[
\mathfrak{C}^n=(\bold{1},\mathfrak{C}^n_{(1)},\mathfrak{C}^n_{(2)},\ldots , \mathfrak{C}^3_{(n+1)},\bold{0}, \bold{0},\ldots),
\] 
where where $\mathfrak{C}^n_{(k)}$ is the degree $k$ structure type of being an embedded $\mathfrak{P}^{k-1}$ in $\mathfrak{P}^n$ equipped with a linear ordering for $k=1,2,\ldots, n+1$. We then have $\mathfrak{C}^n(z)=(1+z)^{n+1}$, thus
\[
c(\mathbb{P}^n)=\mathfrak{C}^n(H)(\text{mod}\hspace{1mm}H^{n+1}).
\]
We now show how the total Chern class of all smooth hypersurfaces in $\mathbb{P}^n$ may be recovered by $\mathfrak{C}^n$ as well.

For this, let $\mathfrak{e}_n(z)$ denote the polynomial which yields the Euler characteristic of smooth hypersurfaces in $\mathbb{P}^n$ upon evaluation at its degree, which by \eqref{ecf} is given by
\[
\mathfrak{e}_n(z)=-\sum_{k=0}^{n-1}\binom{n+1}{k}(-z)^{n-k}.
\] 
Now let $\vartheta$ be the operator on the formal power series ring $\mathbb{Z}[\![z]\!]$ which throws away the linear part of a formal power series and divides the result by $-z$, so that 
\begin{equation}\label{VO}
\vartheta\left(\sum_{n=0}^{\infty}a_nz^n\right)=-\sum_{n=1}^{\infty}a_{n+1}z^n.
\end{equation}
We then have
\[
\mathfrak{e}_n(z)=\vartheta\left(-(1-z)^{n+1}\right),
\]
thus after loosely associating $-(1-z)^{n+1}$ with the total Chern class $(1+H)^{n+1}$ of $\mathbb{P}^n$ (since up to minus signs they contain the same information), $\mathfrak{e}_n(z)$ -- or the top Chern class of a smooth hypersurface in $\mathbb{P}^n$ -- is obtained from the total Chern class of $\mathbb{P}^n$ via the operator $\vartheta$ (after evaluation at its degree). Moreover, for $X$ a smooth hypersurface of degree $d$ in $\mathbb{P}^n$, it is well-known (e.g. via the adjunction formula)
\[
c(X)=\frac{dH(1+H)^{n+1}}{1+dH} (\text{mod}\hspace{1mm}H^{n+1})\in  A_*\mathbb{P}^n,
\]
from which an elementary inductive argument yields
\begin{equation}\label{FCCE}
c(X)=\vartheta^{n-1}\mathfrak{e}_n(d)H+\vartheta^{n-2}\mathfrak{e}_n(d)H^2+\cdots+\vartheta\mathfrak{e}_n(d)H^{n-1}+\mathfrak{e}_n(d)H^{n}.
\end{equation}
As such, it follows that not only the Euler characteristic/top Chern class, but the \emph{total} Chern class of all smooth hypersurfaces in $\mathbb{P}^n$ may be obtained by evaluation at their degree of iterations of $\vartheta$ on $-(1-z)^{n+1}$ (which we loosely associate with the total Chern class of $\mathbb{P}^n$). We now lift this statement to the level of integral structure types. 

So let $\mathscr{C}^n=(\mathscr{C}^{+},\mathscr{C}^-)$ be the integral structure type given by
\begin{equation}\label{AFD}
\mathscr{C}^+=(\bold{0},\mathfrak{C}^n_{(1)},\bold{0},\mathfrak{C}^n_{(3)}, \bold{0},\ldots, \mathfrak{C}^n_{(2m+1)},\ldots), \quad \mathscr{C}^-=(\bold{1},\bold{0},\mathfrak{C}^n_{(2)},\bold{0},\mathfrak{C}^n_{(4)}, \bold{0}, \ldots, \mathfrak{C}^n_{(2m)},\ldots),
\end{equation}
 so that $\mathscr{C}^n$ contains the same information as $\mathfrak{C}^n$, but with the odd-degree components of $\mathfrak{C}^n$ weighted positively and the even-degree components of $\mathfrak{C}^n$ weighted negatively. It then follows
 \[
 \mathscr{C}^n(z)=-(1-z)^{n+1},
 \]
 so that 
 \begin{equation}\label{LGTSD}
 \vartheta\left(\mathscr{C}^n(z)\right)=\mathfrak{e}_n(z).
 \end{equation}

We now define an operator $\mathscr{V}$ which lifts the operator $\vartheta$ to an operator on integral structure types we refer to as \emph{regular}. Recall that in \S\ref{OIST} we referred to a structure type $F$ as regular if for every finite set $S$ we have $\chi(\text{Aut}(x))=1$ for all $x\in F(S)$, and we extend the notion of regularity to integral structure types by saying an integral structure type $\Phi=(\Phi^+,\Phi^-)$ is regular if and only if both its positive and negative parts are regular structure types. Now certainly the structure type $\mathfrak{C}^n$ is regular, since permuting the elements of a $k$-element set $S$ doesn't preserve its orderings. It then follows that the associated integral structure type $\mathscr{C}^n$ is regular, so the operator $\mathscr{V}$ we now define for regular integral structure types will be able to act on $\mathscr{C}^n$. 

Now let $\Phi_{(k)}=(\Phi_{(k)}^+,\Phi_{(k)}^-)$ be a regular integral structure type of degree $k$. The operator $\mathscr{V}$ is then given by
\begin{equation}\label{TO}
\mathscr{V}\Phi_{(k)}=
\begin{cases}
\hspace{1.1cm} \bold{0} \quad \hspace{2.05cm} \text{for $k=0,1$} \\
(\mathscr{X}\Phi_{(k)}^-,\mathscr{X}\Phi_{(k)}^+) \quad \quad \text{otherwise},
\end{cases}
\end{equation}
where $\mathscr{X}$ is the operator on regular stucture types defined in \S\ref{OIST} by equation \eqref{dno}. We now prove the following
\begin{theorem}\label{CMM}
Let $\Phi=(\Phi^+,\Phi^-)$ be a regular integral structure type, $\mathscr{V}$ be the operator on regular integral structure types given by \eqref{TO}, and let $\vartheta$ be the operator on integral power series given by \eqref{VO}. Then
\[
(\mathscr{V}\Phi)(z)=\vartheta \left(\Phi(z)\right)\in \mathbb{Z}[\![z]\!].
\]
\end{theorem}
\begin{proof}
By a classification result for structure types (Proposition~4.6.9 in \cite{Yorgey}, which was first proved by Bergeron et al. in \cite{BergeronEA}), a regular structure type of degree $k$ is isomorphic to $mZ^k$ for some natural number $m\in \mathbb{N}$. As such, once an isomorphism is established with sums of linear orderings, an operator on a regular structure type is determined by its action on $Z^k$ for $k\in \mathbb{N}$. So let $S$ denote a set with $n-1$ elements, and let $\mathscr{X}$ be the operator defined on regular structure types in \S\ref{OIST}. In particular, if $F$ is a regular structure type,
\[
\mathscr{X}F(S)=F_{(n)}(S\sqcup \{\star\})/\mathbb{Z}_n.
\]
We then have
\begin{equation}\label{no}
\mathscr{X}Z^n(S)=Z^n(S\sqcup \{\star\})/\mathbb{Z}_n\cong C^n(S\sqcup \star)\cong Z^{n-1}(S),
\end{equation}
thus $\mathscr{X}Z^n\cong Z^{n-1}$ (the right-most isomorphism comes from the fact that the derivative of length-$n$ cycles $C^n$ is $Z^{n-1}$, which follows from a direct generalization of Example~\ref{ce1}).

Now since $\Phi=(\Phi^+,\Phi^-)$ is a regular integral structure type, by definition its positive and negative components are both regular, so by the aforementioned classification result there exists isomorphisms
\[
\Phi^+\cong (a_0\bold{1},a_1Z,a_2Z^2,\ldots), \quad \text{and} \quad \Phi^-\cong(b_0\bold{1},b_1Z,b_2Z^2,\ldots),
\]
with $a_n,b_n\in \mathbb{N}$. We then have
\[
\Phi(z)=\sum_{n=0}^\infty a_nz^n-\sum_{n=0}^{\infty} b_nz^n=\sum_{n=0}^{\infty}(a_n-b_n)z^n,
\] 
so that
\[
\vartheta(\Phi(z))=\sum_{n=1}^\infty (b_{n+1}-a_{n+1})z^{n}.
\]
On the other hand, by the definition of the operator $\mathscr{V}$ via equation \eqref{TO} we have
\[
\mathscr{V}\Phi\cong\left((\bold{0},b_2Z,b_3Z^2,\ldots),(\bold{0},a_2Z,a_3Z^2,\ldots)\right),
\]
so that
\[
(\mathscr{V}\Phi)(z)=\sum_{n=1}^\infty b_{n+1}z^n-\sum_{n=1}^\infty a_{n+1}z^n=\sum_{n=1}^\infty (b_{n+1}-a_{n+1})z^{n}=\vartheta(\Phi(z)),
\]
as desired.
\end{proof}

We then arrive at the following 
\begin{corollary}
Let $X$ be a smooth hypersurface in $\mathbb{P}^n$ of degree $d$, and let $\mathscr{C}^n$ be the integral structure type given by equation \eqref{AFD}. Then
\[
c(X)=(\mathscr{V}^n\mathscr{C}^n)(d)H+(\mathscr{V}^{n-1}\mathscr{C}^n)(d)H^2+\cdots+(\mathscr{V}^2\mathscr{C}^n)(d)H^{n-1}+(\mathscr{V}\mathscr{C}^n)(d)H^{n}.
\]
In particular, we have
\[
\chi(X)=(\mathscr{V}\mathscr{C}^n)(d).
\]
\end{corollary}
\begin{proof}
By equation \eqref{LGTSD}, equation \eqref{FCCE} may be re-written as
\[
c(X)=\left.\left(\vartheta^n\mathscr{C}^n(z)H+\vartheta^{n-1}\mathscr{C}^n(z)H^2+\cdots +\vartheta^2\mathscr{C}^n(z)H^{n-1}+\vartheta\mathscr{C}^n(z)H^{n}\right)\right|_{z=d}.
\]
Applying Theorem~\ref{CMM} to the above equation then implies the result.
\end{proof}

\bibliographystyle{plain}
\bibliography{IST3}

\begin{thebibliography}{1}

\bibitem{BaezDolan}
John~C. Baez and James Dolan.
\newblock From finite sets to {F}eynman diagrams.
\newblock In {\em Mathematics unlimited---2001 and beyond}, pages 29--50.
  Springer, Berlin, 2001.

\bibitem{BergeronEA}
F.~Bergeron, G.~Labelle, and P.~Leroux.
\newblock {\em Combinatorial species and tree-like structures}, volume~67 of
  {\em Encyclopedia of Mathematics and its Applications}.
\newblock Cambridge University Press, Cambridge, 1998.
\newblock Translated from the 1994 French original by Margaret Readdy, With a
  foreword by Gian-Carlo Rota.

\bibitem{IT}
W.~Fulton.
\newblock {\em Intersection theory}, volume~2 of {\em Ergebnisse der Mathematik
  und ihrer Grenzgebiete. 3. Folge.}
\newblock Springer-Verlag, Berlin, second edition, 1998.

\bibitem{Joyal1}
Andr\'e Joyal.
\newblock Une th\'eorie combinatoire des s\'eries formelles.
\newblock {\em Adv. in Math.}, 42(1):1--82, 1981.

\bibitem{Joyal2}
Andr\'e Joyal.
\newblock R\`egle des signes en alg\`ebre combinatoire.
\newblock {\em C. R. Math. Rep. Acad. Sci. Canada}, 7(5):285--290, 1985.

\bibitem{Morton}
Jeffrey Morton.
\newblock Categorified algebra and quantum mechanics.
\newblock {\em Theory Appl. Categ.}, 16:No. 29, 785--854, 2006.

\bibitem{Teitler}
Z.~Teitler.
\newblock An informal introduction to computing with chern classes in algebraic
  geometry.
\newblock Author's VIGRE project.
\newblock https://math.boisestate.edu/~zteitler/math/expository/chern.pdf.

\bibitem{Yeh1}
Yeong~Nan Yeh.
\newblock The calculus of virtual species and {${\bf K}$}-species.
\newblock In {\em Combinatoire \'enum\'erative ({M}ontreal, {Q}ue.,
  1985/{Q}uebec, {Q}ue., 1985)}, volume 1234 of {\em Lecture Notes in Math.},
  pages 351--369. Springer, Berlin, 1986.

\bibitem{Yorgey}
Brent~Abraham Yorgey.
\newblock {\em Combinatorial species and labelled structures}.
\newblock ProQuest LLC, Ann Arbor, MI, 2014.
\newblock Thesis (Ph.D.)--University of Pennsylvania.

\end{thebibliography}

\end{document}